\documentclass[11pt]{article}
\usepackage{graphicx,amsmath,amsfonts,bm,cite,epsfig,epsf,url,color}
\usepackage{bm}
\usepackage{subfigure}
\usepackage{dsfont}
\usepackage{fullpage}
\usepackage[small,bf]{caption}
\usepackage{times}
\usepackage{algorithm}
\usepackage{algorithmic}
\newtheorem{theorem}{Theorem}[section]
\newtheorem{lemma}[theorem]{Lemma}

\newcommand{\ten}[1]{\mathcal{#1}}
\newcommand{\fold}{\operatorname{fold}}
\newcommand{\SVT}{\operatorname{SVT}}
\newcommand{\prox}{\operatorname{prox}}
\newcommand{\eye}{\operatorname{eye}}
\newcommand{\St}{\operatorname{St}}
\newcommand{\diag}{\operatorname{diag}}
\newcommand{\sgn}{\operatorname{sgn}}
\newcommand{\vecc}{\operatorname{vec}}
\newcommand{\ser}[1]{\mathfrak{#1}}

\numberwithin{equation}{section}

\def \endprf{\hfill {\vrule height6pt width6pt depth0pt}\medskip}
\newenvironment{proof}{\noindent {\bf Proof} }{\endprf\par}

\title{Parallel Active Subspace Decomposition for Scalable and Efficient
Tensor Robust Principal Component Analysis}

\author{Jonathan Q. Jiang and Michael K. Ng\\
 \vspace{-.1cm}\\
  Department of Mathematics, Hong Kong Baptist University
}

\date{}

\begin{document}

\maketitle

\vspace{-0.3in}

\begin{abstract}
Tensor robust principal component analysis (TRPCA) has received a substantial amount of attention in various fields. Most existing methods, normally relying on tensor nuclear norm minimization, need to pay an expensive computational cost due to multiple singular value decompositions (SVDs) at each iteration. To overcome the drawback, we propose a scalable and efficient method, named Parallel Active Subspace Decomposition (PASD), which divides the unfolding along each mode of the tensor into a columnwise orthonormal matrix (active subspace) and another small-size matrix in parallel. Such a transformation leads to a nonconvex optimization problem in which the scale of nulcear norm minimization is generally much smaller than that in the original problem. Furthermore, we introduce an alternating direction method of multipliers (ADMM) method to solve the reformulated problem and provide rigorous analyses for its convergence and suboptimality. Experimental results on synthetic and real-world data show that our algorithm is more accurate than the state-of-the-art approaches, and is orders of magnitude faster.
\end{abstract}


\section{Introduction}
\label{sec1}

\noindent The prevalence of multidimensional data, such as multichannel images and videos,  in modern society, has revived our interest in the study for tensor decomposition, completion and recovery in last decade. Tensor, as higher-order generalization of vector and matrix, is able to take full advantage of the multilinear structure of the data and thus to provide better understanding and higher precision in signal processing~\cite{Cichocki2014}, computer vision~\cite{Tao2007,Liu2013}, data mining~\cite{Sun2009,Morup2011} and machine learning~\cite{Li2008,Signoretto2014}.

Multidimensional data analysis traditionally relies on tensor decomposition~\cite{Kolda2005}, which normally takes two popular forms, CANDECOMP/PARAFAC (CP) decomposition~\cite{Harshman1969} and Tucker decomposition~\cite{Tucker1966}. Originated in the fields of psychometrics and chemometrics, these decompositions are now used in a wide range of application areas (see~\cite{Kolda2005} for a comprehensive review). Owing to many factors, including the malfunctions in the acquisition process, loss of information, and expensive experiments, the multidimensional data is probably incomplete in many applications, which prevents both types of tensor decompositions from achieving satisfactory results. To address tensor data with missing values, two extended models called weighted Tucker~\cite{Filipovi2013} and weighted CP decomposition~\cite{Acar2010}, have been recently proposed and successfully applied to EEG data analysis and image inpainting.

In reality, the intrinsic structures of the real data have been found to be actually low-rank, even if themselves may not be. Unlike matrices, the rank of a specific tensor is NP-hard to estimate in general~\cite{Hastad2006} and there exists no explicit expression for its tightest convex envelop so far. In the seminar work~\cite{Liu2013}, the first convex approximation of tensor rank named tensor trace norm was given as a weighted combination of the trace norms of all matrices unfolded along each mode. Soon after, a large number of algorithms~\cite{Signoretto2010,Gandy2011,Yang2013,Signoretto2014} were proposed for the low-rank tensor completion (LRTC) problem, i.e., learning a low-rank tensor from partially observed data, on the basis of tensor trace norm minimization.

In this paper, we are particularly interested in another branch of the low-rank tensor recovery problem, namely Tensor Robust Principal Component Analysis (TRPCA).  
More precisely, we aim to split a noisy and fully observed tensor into a low-rank component that captures its underlying low-dimensional structure and a sparse component that contains the gross errors. This problem is essentially a tensor version of Robust Principal
Component Analysis (RPCA) in matrix case~\cite{Candes2011}. Compared to the LRTC problem, investigations of the TRPCA problem are relatively limited and only can be found in a few papers~\cite{Li2010,Tan2013,Goldfarb2013,Huang2015}. All the methods employed the tensor trace norm and tensor $\ell_1$ norm to enforce the low-rankness and the sparsity of the two components respectively and depended on an alternating direction method of multipliers (ADMM) scheme, which suffered from a heavy computational burden due to the multiple singular value decompositions (SVDs) conducted in each iteration.

To address this issue, we propose an efficient and scalable method called Parallel Active Subspace Decomposition (PASD) in this paper. It is quite interesting and innovative from the following perspectives.
\begin{itemize}
  \item Our PASD method simultaneously decomposes the unfolding along
each mode of the tensor into a columnwise orthonormal matrix, e.g., active subspace~\cite{Liu2012}, and another small-size matrix. The computational cost is significantly reduced, since the trace norms of the unfoldings are equivalently replaced by those of some smaller-size matrices.
  \item We introduce an effective and efficient ADMM algorithm to solve the nonconvex optimization problem, which seems particularly suitable for large-scale problems.
  \item We conduct rigorous analyses for the convergence and suboptimality of our algorithm.
  \item Experimental results show that our PASD method is much more accurate than the state-of-the-art approaches, especially when the Tucker rank is large, and is orders of magnitude faster.
\end{itemize}

We begin with a brief review of tensor basics and related works in Section~2. Section~3 gives our PASD model and its corresponding ADMM algorithm. In Section~4, we present the theoretical analyses for the convergence and suboptimality of our algorithm. Finally, we report experimental results in Section~5 and draw the conclusions in Section~6.

\section{Notations and Preliminaries}
\label{sec2}

Matrices are denoted by uppercase letters, e.g., $X$, and tensors by calligraphic letters, e.g., $\ten{X}$ throughout the paper.

\subsection{Tensor Basics}
\label{sec2:sub1}

The order of a tensor is the number of dimensions, also known as ways or modes.
Given a $N$-order tensor $\ten{X} \in \Re^{I_1 \times \cdots \times I_N}$, a fiber is a column vector defined by fixing every index of $\ten{X}$ but one. The
mode-$n$ unfolding or matricization is the matrix denoted by
$\ten{X}_{(n)} \in \Re^{I_n \times \prod_{m \neq n} I_m}$ that is obtained by arranging the mode-$n$ fibers to be the columns of the matrix. The mode-$n$ product of a tensor $\ten{X} \in \Re^{I_1 \times \cdots \times I_N}$ with a matrix $U \in \Re^{J \times I_n}$ is defined as $(\ten{X} \times_n U)_{i_1 \cdots i_{n-1}ji_{n+1} \cdots i_N} = \sum_{i_n = 1}^{I_n} x_{i_1 \cdots i_N} u_{j i_n}$. The inner product of two tensors $\ten{X}, \ten{Y} \in \Re^{I_1 \times \cdots \times I_N}$ is defined as the sum of the product of their entries, i.e., $\langle \ten{X}, \ten{Y}\rangle = \sum_{i_1 \cdots i_N} x_{i_1 \cdots i_N} y_{i_1 \cdots i_N}$, and the Frobenius norm of $\ten{X}$ is defined as $\|\ten{X}\|_F = \sqrt{\langle \ten{X}, \ten{X}\rangle}$. The $\ell_1$ norm and $\ell_{\infty}$ norm of a tensor $\ten{X}$ are defined by its vectorization, i.e, $\|\ten{X}\|_{1} = \|\vecc(\ten{X})\|_{1}$ and $\|\ten{X}\|_{\infty} = \|\vecc(\ten{X})\|_{\infty}$ respectively. The mode-$n$ rank of $\ten{X}$ is the column rank of $\ten{X}_{(n)}$. The set of $N$ mode-$n$ ranks $(r_1, \cdots, r_N)$ of a tensor $\ten{X}$ is called its multilinear rank or Tucker rank.

\subsection{Tensor Decompositions and Ranks}
\label{sec2:sub2}

The CP decomposition~\cite{Harshman1969} approximates a tensor as $\ten{X} \approx \sum_{i=1}^r \lambda_n \bm{a}^{(1)}_i \circ \bm{a}^{(2)}_i \circ \cdots \circ \bm{a}^{(N)}_i$ where $\circ$ stands for the outer product of two vectors, $\lambda_n \in \Re$ and $\bm{a}^{(n)}_i \in \Re^{I_n}$ for $i = 1, \cdots, r$ and $n = 1, \cdots, N$. The rank of $\ten{X}$ is the smallest value of $r$ such that the approximation holds with equality. The Tucker decomposition~\cite{Tucker1966} is another factorization that approximates a tensor as $\ten{X} \approx \ten{C} \times_1 U_1 \times_2 U_2 \cdots \times_N U_N$ where $\ten{C} \in \Re^{r_1 \times \cdots \times r_N}$ is the core tensor and $U_n \in \Re^{I_n \times r_n}, n = 1, \cdots, N$ are the factor matrices. The mode-$n$ rank of $\ten{X}$ is the column rank of $\ten{X}_{(n)}$. The set of $N$ mode-$n$ ranks $(r_1, \cdots, r_N)$ of a tensor $\ten{X}$ is called its multilinear rank or Tucker rank.

\subsection{Related work}

Given a tensor data $\ten{T}$, the TRPCA problem can be mathematically represented by
\begin{equation}
\min_{\ten{X},\,\ten{E}} \sum_{n=1}^N \lambda_n\|\ten{X}_{(n)}\|_{\ast}+\|\ten{E}\|_{1},\quad \textup{s.t.},\, \ten{T} = \ten{X} + \ten{E},
\label{eq1}
\end{equation}
where $\ten{X}$ and $\ten{E}$ are the low-rank and sparse components, $\|\ten{X}_{(n)}\|_{\ast}$ denotes the trace norm of the unfolding $\ten{X}_{(n)}$, i.e., the sum of its singular values, and $\{\lambda_n\}$ are prespecified weights. Note that problem (\ref{eq1}) is very difficult to solve because of the interdependent matrix trace norm terms.

The rank sparsity tensor decomposition (RSTD) algorithm~\cite{Li2010} applies variable-splitting to both $\ten{X}$ and $\ten{E}$, and utilizes a classic Block Coordinate Descent (BCD) approach to solve an unconstrained problem obtained by relaxing all the constraints as quadratic penalty terms. However, this method has many parameters to tune and does not have a iteration complexity guarantee. The Multi-linear Augmented Lagrange Multiplier (MALM) Method~\cite{Tan2013} is based on the ADMM algorithm and decomposes problem (\ref{eq1}) into $N$ independent standard RPCA instances. This relaxation makes the final solution hard to be optimal since consistency among the auxiliary variables is not considered. The Higher-order RPCA (HoRPCA) approach~\cite{Goldfarb2013} is also an ADMM method that conducts variable-splitting on $\ten{X}$ purely and reformulates problem (\ref{eq1}) as
\begin{align}
\label{eq2}\min_{\ten{Z}_n, \ten{E}} & \sum_{n=1}^N \lambda_n\|\ten{Z}_{n, (n)}\|_{\ast}+\|\ten{E}\|_{1}, \\
\textup{s.t.}, &\,\,\, \ten{T} = \ten{Z}_n + \ten{E},\,\,\,\forall n \in \mathds{N} \nonumber
\end{align}
where $\{\ten{Z}_n\}$ are the auxiliary variables and $\mathds{N}$ is the index set $\{1, 2, \cdots, N\}$. Note that equality among the $\ten{Z}_n$s is enforced implicitly by the constraints, so that additional auxiliary variables for $\ten{E}$ as in~\cite{Li2010,Tan2013} are not required.
Unfortunately, all the three approaches involve multiple SVDs of the unfoldings in each iteration, and thus are prone to suffer from expensive computational cost when the scale of the TRPCA problem is very large.

\section{Our Method}
\label{sec3}

In this section, we first introduce the PASD model for problem (\ref{eq2}) and then propose an efficient ADMM iterative scheme to solve the new nonconvex optimization problem.

\subsection{Parallel Active Subspace Decomposition}
\label{sec3:sub1}

It is well-known that matrix factorization is one of the most useful tools in high-dimensional data analysis, on account of its high accuracy, scalability and flexibility to incorporating side information. Given a large-size matrix can be approximated by the product of two matrices with much smaller size. Inspired by the previous work~\cite{Liu2012}, we decompose the unfoldings in problem (\ref{eq2}) as
\begin{equation}
\ten{Z}_{n, (n)} =  U_n V_n, \quad \textup{s.t.}, U_n \in \St(I_n, R_n), \,\,\, \forall n \in \mathds{N}\nonumber
\end{equation}
where $\St(I_n, R_n)$ denotes the Stiefel manifold, i.e., the set of columnwise orthonormal matrices of size $I_n \times R_n$, and $r_n \leq R_n \ll I_n$ is a given upper bound on the mode-$n$ rank of $\ten{T}$. The matrices $\{U_n\}$ are called active subspaces in~\cite{Liu2012}, since the underlying principle behind such a decomposition is similar to the famous active set~\cite{Nocedal2006}. Due to the orthonormality of $U_n$s, we have that $\|\ten{Z}_{n, (n)}\|_{\ast} = \|U_n V_n\|_{\ast} = \|V_n\|_{\ast}, \forall n \in \mathds{N}$, and problem (\ref{eq2}) can be rewritten as
\begin{align}
\label{eq3}\min_{U_n, V_n, \ten{E}} & \,\,\, \sum_{n=1}^N \lambda_n\|V_n\|_{\ast}+\|\ten{E}\|_{1}, \\
\textup{s.t.}, & \, \ten{T} = \fold_n(U_n V_n) + \ten{E}, \,U_n \in \St(I_n, R_n), \,\forall n \in \mathds{N}, \nonumber
\end{align}
where $\fold_n(A)$ returns the tensor $\ten{A}$ such that $\ten{A}_{(n)} = A$. In (\ref{eq3}), we carry out the active subspace decomposition along all modes in parallel and this is the main reason why we name our method. Note that there is a related work~\cite{Xu2015} which also makes use of parallel matrix factorization. But our study departs from it on the following two fronts.
The problem considered in that work is obviously of different nature. That work mainly concentrated on the LRTC problem while ours focuses on the TRPCA problem.
Besides, the standard low-rank matrix factorizations were used in that work, with no orthonormal constraint. More importantly, the trace norms are preserved in problem (\ref{eq3}), which has been shown to be very helpful to the robustness of algorithms against outliers and non-Gaussian noise~\cite{Liu2012,Okutomi2012,Cabral2013}.

\subsection{ADMM Algorithm}
\label{sec3:sub2}

The ADMM method is very efficient for some convex or non-convex programming
problems from various applications~\cite{Boyd2011}. Therefore, we propose an
ADMM algorithm to solve problem (\ref{eq3}).

The partial augmented Lagrangian function for problem (\ref{eq3}) is given by
\begin{align}
& \mathcal{L}_{\mu} (U_1 \cdots U_N, V_1 \cdots V_N, \ten{E}, \ten{Y}_1 \cdots \ten{Y}_N) \nonumber\\
  & = \sum_{n=1}^N \Big(\lambda_n \|V_n\|_{\ast} + \langle \ten{Y}_n, \ten{T} - \fold_n(U_n V_n) - \ten{E}\rangle \nonumber\\
& \,\,\,+ \frac{\mu}{2}\|\ten{T} - \fold_n(U_n V_n) - \ten{E}\|^2_F\Big) + \|\ten{E}\|_{1},
\label{eq4}
\end{align}
where $\ten{Y}_n, \forall n \in \mathds{N}$ are the tensors of Lagrange multipliers and $\mu$ is the penalty parameter. We give an iterative scheme to minimize $\mathcal{L}_{\mu}$ with respect to $\{U_n\}$, $\{V_n\}$, $\ten{E}$ successively.

By removing the terms irrelevant to $U_n$ and adding some
proper terms independent on $U_n$, problem (\ref{eq4}) with respect to $U_n$ can be simplified as
\begin{equation}
\min_{U_n} \|U_n V^k_n - G^k_n\|^2_F, \quad \textup{s.t.},\, U_n \in \textup{St}(I_n, R_n),
\label{eq9}
\end{equation}
where $G^k_n = \ten{T}_{(n)} - \ten{E}^k_{(n)} + {\ten{Y}^k_{n, (n)}}/{\mu^k}$. This is actually the well-known orthogonal procrustes problem~\cite{Higham1995}. Suppose the SVD of the matrix $G^k_n (V^k_n)^T$ is $G^k_n (V^k_n)^T = \widehat{U^k_n} \widehat{\Sigma^k_n} (\widehat{V^k_n})^T$, and the optimal solution can
be given by
\begin{equation}
U^{k+1}_n = \widehat{U^k_n} (\widehat{V^k_n})^T.
\label{eq10}
\end{equation}

By the similar way, problem (\ref{eq4}) with respect to $V_n$ can be reformulated as
\begin{equation}
\min_{V_n} \lambda_n \|V_n\|_{\ast} + \frac{\mu^k}{2}\|U^{k+1}_nV _n - G^k_n\|^2_F.
\label{eq11}
\end{equation}
Considering that $U^{k+1}_n \in \St(I_n, R_n)$, problem~(\ref{eq10}) is equivalent to
\begin{equation}
\min_{V_n} \lambda_n \|V_n\|_{\ast} + \frac{\mu^k}{2}\|V _n - (U^{k+1}_n)^TG^k_n\|^2_F,
\label{eq12}
\end{equation}
which has a closed-form solution
\begin{equation}
V^{k+1}_n = \SVT_{\frac{\lambda_n}{\mu^k}}\Big(({U^{k+1}_n})^T G^k_n\Big).
\label{eq13}
\end{equation}
The singular value thresholding (SVT) operator is defined
by $\SVT_\tau(X) = U\diag(\max(\Sigma - \tau,0))V^T$ where the SVD of matrix $X$ is $X = U \Sigma V^T$ and $\max(\cdot,\cdot)$ should be understood element-wise.

Fixing $\{U_n\}$ and $\{V_n\}$, we can update $\ten{E}$ by solving
\begin{equation}
\min_{\ten{E}} \|\ten{E}\|_1 + \frac{\mu^k}{2}\sum_{n=1}^N\|\ten{E} - \ten{H}^k_n\|^2_F,
\label{eq14}
\end{equation}
where $\ten{H}^k_n = \ten{T} - \fold_n(U^{k+1}_n V^{k+1}_n) + \ten{Y}^k_n/\mu^k$.
As indicated in~\cite{Goldfarb2013}, problem (\ref{eq13}) has the following closed-form solution
\begin{equation}
\ten{E}^{k+1} = \prox_{\frac{1}{\mu^k N}}\left(\frac{1}{N}\sum_{n=1}^N\ten{H}^k_n\right),
\label{eq15}
\end{equation}
where $\prox_\tau(\cdot)$ denotes the shrinkage operator, namely, $\prox_\tau(x) =
\sgn(x)\max(|x| - \tau,0)$.

Summarizing the above analysis, we obtain an ADMM
algorithm for problem (\ref{eq3}), as outlined in Algorithm \ref{alg1}. Note that the algorithm can be further accelerated by adaptively changing $\mu$ in each iteration (see line 7 in Algorithm~\ref{alg1}).

\begin{algorithm}[!t]
\caption{PASD: Solving (\ref{eq3}) via ADMM}
\label{alg1}
\textbf{Input:} $\ten{T}$, $(R_1, \cdots, R_n)$, $\lambda$ and $\varepsilon$.\\
\textbf{Initialize:} $U^0_n = \eye(I_n, R_n)$, $V^0_n = 0$, $\ten{Y}^0_n = 0$, $n = 1, \cdots, N$, $\ten{E}^0 = 0$, $\mu^0 = 10^{-4}$, $\mu_{\max} = 10^{10}$ and $\rho = 1.1$.\\\vspace{-0.4cm}
\begin{algorithmic}[1]
\WHILE {not converged}
\STATE {Update $U^{k+1}_n$ by (\ref{eq10}).}
\STATE {Update $V^{k+1}_n$ by (\ref{eq13}).}
\STATE {Update $\ten{E}^{k+1}$ by (\ref{eq15}).}
\STATE {Compute $\ten{Z}^{k+1}_n = \fold_n(U^{k+1}_nV ^{k+1}_n)$.}
\STATE {Update the multipliers $\ten{Y}^{k+1}_n$ by\\ $\ten{Y}^{k+1}_n=\ten{Y}^{k}_n+\mu^{k}(\ten{T}-\ten{Z}^{k+1}_n-\ten{E}^{k+1})$.}
\STATE {Update $\mu^{k+1}$ by $\mu^{k+1}=\textup{min}(\rho\mu^{k},\,\mu_{\max})$.}
\STATE {Check the convergence condition, \\
$\|\ten{T}-\ten{Z}^{k+1}_n-\ten{E}^{k+1}\|_{\infty} <\varepsilon, \,\,\, \forall n \in \mathds{N}$.}
\ENDWHILE
\end{algorithmic}
\textbf{Output:} $\ten{X} = \frac{1}{\sum_{n=1}^N \alpha_n}\sum_{n=1}^N \alpha_n \ten{Z}_n$.
\end{algorithm}

\section{Theoretical Analysis}
\label{sec4}

In this section, we will provide complexity analysis for Algorithm~\ref{alg1} and present its several theoretical properties.

\subsection{Complexity Analysis}
\label{sec4:sub1}

The running time of Algorithm~\ref{alg1} is dominated by conducting SVD on much smaller matrices of sizes $I_n \times R_n$  and $R_n \times \prod_{m \neq n} I_m, n \in \mathds{N}$. The time complexity of performing SVD in (\ref{eq10}) and in (\ref{eq13}) are $O(R^2_n I_n)$ and $O(R^2_n \prod_{m \neq n}I_m)$, respectively. The time complexity of some matrix multiplications is $O(R_n \prod_n I_n)$. Therefore, the total time complexity of Algorithm~\ref{alg1} is $O(T\sum_n (R^2_n I_n + R^2_n \prod_{m \neq n}I_m + R_n \prod_n I_n))$ where $T$ is the number of iterations. Without loss of generality, we assume that the time complexity of Algorithm~\ref{alg1} in each iteration is only $O(NRI^N)$ provided that the sizes of the input tensors are $I_1 = \cdots = I_n = I$ and the given ranks are $R_1 = \cdots = R_n = R$ ($R \ll I$). Recall that the complexities of most existing approaches, e.g. MALM~\cite{Tan2013}, SNN~\cite{Huang2015} and HoRPCA~\cite{Goldfarb2013}, in each iteration are all $O(NI^{N+1})$. Thus, our PASD method is much more efficient, as shown in the experiments later.

\subsection{Convergence Analysis}

Next, we check the convergence of our proposed algorithm. In fact, Algorithm~\ref{alg1} can stop within a finite number of iterations, as shown in the following theorem.

\begin{theorem}
Let $(\{U^k_1 \cdots U^k_N\}, \{V^k_1 \cdots V^k_N\}, \ten{E}^k)$ be a sequence generated by Algorithm~\ref{alg1}, then we have that
\begin{description}
  \item[(I)] The sequences $\{V^k_n\}$, $\{U^k_n V^k_n\}, \forall n \in \mathds{N}$ and $\{\ten{E}^k\}$ are Cauchy sequences respectively.
  \item[(II)] $(U^k_n, V^k_n, \ten{E}^k)$ is a feasible solution to problem (\ref{eq3}) in a sense that
  \begin{displaymath}
  \lim\limits_{k \rightarrow \infty} \| \ten{T} - \fold_n(U^k_n V^k_n) - \ten{E}^k\|_{\infty}
  < \varepsilon, \, \forall n \in \mathds{N}.
  \end{displaymath}
\label{the1}
\end{description}
\end{theorem}
The proof of Theorem~\ref{the1} is quite similar to those in~\cite{Lin2009,Liu2012}. We
first introduce two additional groups of auxiliary Lagrangian multipliers,
\begin{align}
\label{eq16} \hat{\ten{Y}}^{k+1}_n &= \ten{Y}^{k}_n + \mu^k (\ten{T} - \fold_n(U^{k+1}_n V^{k+1}_n) - \ten{E}^{k}),\\
\label{eq17} \bar{\ten{Y}}^{k+1}_n & = \ten{Y}^{k}_n + \mu^k (\ten{T} - \fold_n(U^{k+1}_n V^{k}_n) - \ten{E}^{k}),
\end{align}
for $n \in \mathds{N}$ and study the boundedness of them as well as some variables in Algorithm~\ref{alg1}, which are summarized in the following lemma.
\begin{lemma}
The sequences $\{\ten{Y}^k_n\}$, $\{\hat{\ten{Y}}^k_n\}$, $\{\bar{\ten{Y}}^k_n\}$, $\{\ten{E}^k\}$, $\{V^k_n\}$, and $\{U^k_n V^k_n\}$, $\forall n \in \mathds{N}$ are all bounded.
\label{lem1}
\end{lemma}
We can then use this lemma to prove Theorem~\ref{the1}. The detailed proof is given in Appendix \ref{appA} and \ref{appB}.

\subsection{Suboptimality Analysis}
\label{sec4:sub3}

In this subsection, we attempt to show that it is possible to prove the local optimality of the solution produced by Algorithm~\ref{alg1}. In other words, we want to investigate the gap between the true minimum and the minimal value of the objective function achieved by our proposed algorithm.

Let $k^{\ast}$ be the number of iterations when Algorithm~\ref{alg1} stops, and $U^{\ast}_n= U^{k^{\ast}+1}_n$, $V^{\ast}_n= V^{k^{\ast}+1}_n$, and $\ten{E}^{\ast}= \ten{E}^{k^{\ast}+1}$ respectively. Besides, $\ten{Y}^{\ast}_n$ and $\hat{\ten{Y}}^{\ast}_n$ denote the Lagrange multipliers $\ten{Y}^{k^{\ast}+1}_n$ and $\hat{\ten{Y}}^{k^{\ast}+1}_n$ corresponding to $(\{U^{\ast}_n\}, \{V^{\ast}_n\}, \ten{E}^{\ast})$. Then we have the following lemma whose proof can be found in Appendix ~\ref{appC}.

\begin{lemma}
Given the solution $(\{U^{\ast}_n\}, \{V^{\ast}_n\}, \ten{E}^\ast)$ generated by Algorithm~\ref{alg1}, the following conclusion holds
\begin{align}
\sum_{n=1}^N \lambda_n\|V_n\|_{\ast}+\|\ten{E}\|_1 \geq \sum_{n=1}^N \lambda_n\|V^{\ast}_n\|_{\ast}+\|\ten{E}^{\ast}\|_1 \nonumber\\
\label{eq18} +\sum_{n=1}^N \langle \ten{Y}^{\ast}_n - \hat{\ten{Y}}^{\ast}_n, \ten{E} - \ten{E}^{\ast}\rangle - \sum_{n=1}^N \lambda_n I_n \prod_{m \neq n} I_m \varepsilon,
\end{align}
for any feasible solution $(\{U_n\}, \{V_n\}, \ten{E})$ to problem (\ref{eq3}).
\label{lem2}
\end{lemma}
To reach the global optimality of problem (\ref{eq3}), we are required to show that the term $\sum_{n=1}^N \langle \ten{Y}^{\ast}_n - \hat{\ten{Y}}^{\ast}_n, \ten{E} - \ten{E}^{\ast} \rangle$ almost surely vanishes. According to the proofs of Theorem~\ref{the1} and Lemma~\ref{lem2} (see the Supplementary Materials), we can conclude that
\begin{align}
\Big\|\sum_{n=1}^N \left(\ten{Y}^{\ast}_n - \hat{\ten{Y}}^{\ast}_n\right)\Big\|_{\infty}  \leq  & \,\,\Big\|\sum_{n=1}^N \ten{Y}^{\ast}_n\Big\|_{\infty} + \sum_{n=1}^N \|\hat{\ten{Y}}^{\ast}_n\|_{\infty} \nonumber\\
\leq & \,\, 1 + \sum_{n=1}^N \lambda_n
\end{align}
which means that $\sum_{n=1}^N (\ten{Y}^{\ast}_n - \hat{\ten{Y}}^{\ast}_n)$ is bounded. By
setting the parameter $\rho$ to be relatively small (e.g., $\rho = 1.1$ as suggested in~\cite{Liu2012}), $\sum_{n=1}^N (\ten{Y}^{\ast}_n - \hat{\ten{Y}}^{\ast}_n)$ can be sufficiently small. Let $\epsilon = \|\sum_{n=1}^N (\ten{Y}^{\ast}_n - \hat{\ten{Y}}^{\ast}_n)\|_{\infty}$, then we have the following theorems.

\begin{theorem}
Let $f^g$ be the globally optimal objective function value of (\ref{eq3}), and $f^\ast$ be the objective function value of (\ref{eq3}) generated by Algorithm~\ref{alg1}. We have that
\begin{equation}
f^\ast \leq f^g + c \epsilon + \sum_{n=1}^N \lambda_n I_n \prod_{m \neq n} I_m \varepsilon
\label{eq20}
\end{equation}
where $c$ is a constant defined by
\begin{displaymath}
c = \frac{1}{\mu^0 N^2} \sum_{n=1}^N I_n \prod_{m \neq n} I_{m} \left(\frac{\rho(1+\rho)}{\rho-1} + \frac{1}{2\rho^{k^\ast}}\right) + \|\mathcal{T}\|_1
\end{displaymath}
\label{the2}
\end{theorem}

\begin{theorem}
Suppose $(\ten{X}^o, \ten{E}^o)$ is an optimal solution to problem (\ref{eq1}), the Tucker rank of $\ten{X}^o$ is $(r_1, \cdots, r_N)$, and $f^o= \sum_{n=1}^N \alpha_n \|\ten{X}^o_{(n)}\|_{\ast} + \lambda\|\ten{E}^o\|_1$. Let $f^\ast$ be the objective function value of (\ref{eq3}) returned by Algorithm~\ref{alg1}, then we have
\begin{align}
& f^o \leq f^\ast \leq f^o + c \epsilon + \sum_{n=1}^N \lambda_n I_n \prod_{m \neq n} I_m \varepsilon \nonumber\\
\label{eq21} & + \sum_{n=1}^N \lambda_n \left(\sqrt{I_n \prod_{m \neq n} I_m} - 1 \right)\sigma_{(n)}^{R_n+1} \max(r_n - R_n, 0),
\end{align}
where $\sigma_{(n)}^{i}$ is the $i$th largest singular value of $\ten{X}^o_{(n)}$.
\label{the3}
\end{theorem}
The proofs of Theorem~\ref{the2} and~\ref{the3} can be found in Appendix ~\ref{appD} and ~\ref{appE}. These two theorems reduce to their counterparts, Theorem 3.1 and 3.2 in~\cite{Liu2012}, provided that the tensors are two-dimensional matrices.

\section{Experiments and Discussions}

In this section, we systematically evaluate the effectiveness and efficiency of our PASD method on synthetic and real-world data. All the experiments are performed with \textsc{Matlab} 8.1 on an Intel Xeon E5-2620 workstation with 2.0-GHz CPU and 24-GB memory.

\begin{table*}[!t]
\caption{RSE and running time (seconds) comparison on synthetic tensor data.}
\label{tab1}
\vspace{0.5em}
\centering
\begin{tabular}{ccccccccc}
\multicolumn{9}{c}{(a) Tensor size: 100 $\times$ 100 $\times$ 100}\\
\hline
 & \multicolumn{2}{c}{RPCA} & \multicolumn{2}{c}{MALM} & \multicolumn{2}{c}{SNN} & \multicolumn{2}{c}{PASD}\\
$\rho_n$ & RSE & Time & RSE & Time & RSE & Time & RSE & Time \\
\hline
5\% & 1.87e-7	& 35.50 & 1.80e-7 & 164.26 & 1.28e-7 & 171.21 & 1.12e-7	& 54.89\\
10\% & 3.92e-7	& 54.66 & 4.69e-5 & 169.97 & 9.91e-7 & 185.48 & 1.28e-7	& 55.95\\
20\% & 8.08e-4	& 65.78 & 1.98e-3 & 189.67 & 1.22e-6 & 207.74 & 1.44e-7	& 57.65\\
\hline
\multicolumn{9}{c}{(a) Tensor size: 50 $\times$ 50 $\times$ 50 $\times$ 50}\\
\hline
 & \multicolumn{2}{c}{RPCA} & \multicolumn{2}{c}{MALM} & \multicolumn{2}{c}{SNN} & \multicolumn{2}{c}{PASD}\\
$\rho_n$ & RSE & Time & RSE & Time & RSE & Time & RSE & Time \\
\hline
5\% & 4.57e-3 & 300.34 & 3.50e-3  &	1191.80 & 3.30e-3 &1750.07 & 2.99e-3 & 469.46\\
10\% & 1.20e-2 & 453.26 & 7.74e-3 & 1777.02 & 9.02e-3 & 1782.93 & 8.33e-3 & 483.37 \\
20\% & 9.63e-1 & 611.13 & 9.39e-1 & 2344.38 & 4.72e-2 & 2286.70 & 4.27e-2 & 544.28 \\
\hline
\end{tabular}
\end{table*}

\subsection{Synthetic Tensor Recovery}

We generate a low-rank tensor $\ten{T}_0 \in \Re^{I_1 \times \cdots \times I_N}$, which is
used as ground truth, by the Tucker decomposition model. As described in~\cite{}, we draw the entries of the core tensor $\ten{C} \in \Re^{r_1 \times \cdots \times r_N}$ from the standard normal distribution $\mathcal{N}(0, 1)$ and multiply each mode of the core tensor by an columnwise orthonormal factor matrix $U_n \in \Re^{I_n \times r_n}$ drawn from the Haar measure. All generated tensors were verified to have the desired Tucker rank. A random fraction $\rho_n$ of the tensor elements were corrupted by additive i.i.d. noise from the uniform distribution $\mathcal{U}[-1, 1]$.

We recover the low-rank tensor by our PASD algorithm and compare it with two state-of-the-art approaches, MALM~\cite{Tan2013} and SNN~\cite{Huang2015}. We also conduct RPCA~\cite{Candes2011} on the unfoldings along all the modes and report the best result.  Without loss of generality, we set the size of tensor to be $100 \times 100 \times 100$ and $50 \times 50 \times 50 \times 50$ respectively and fix the Tucker rank to be $r_n = 10, \forall n \in \mathds{N}$.  We set $\varepsilon = 10^{-5}$ and maxiter = 1000 for all the algorithms. The parameter $\lambda$ for RPCA and MALM are set to be the default values. For SNN and PASD, the parameters $\lambda_n$ are set as $\frac{\sqrt{\max(I_n, \prod_{j \neq i}I_j)}}{N}$. The upper bound of Tucker ranks are chosen as $R_n = R = \lfloor1.2r\rfloor, \forall n \in \mathds{N}$ for PASD. The relative square error (RSE) of the recovered tensor $\ten{X}$ is measured by $\textup{RSE} = \|\ten{X} - \ten{T}_0\|_F/\|\ten{T}_0\|_F$.

The average results (RSE and computational time) of ten independent runs are summarized in Table~\ref{tab1}, where $\rho_n$ is set to 5\%, 10\% or 20\%. We can see clearly that the PASD algorithm always outperforms the other approaches in terms of RSE and efficiency
in all the cases. In particular, it can yield much more accurate solutions using less time for synthetic tensors of size $50 \times 50 \times 50 \times 50$ which is much more difficult to be recovered due to the relatively large ratio of the Tucker rank and tensor size. The empirical performance of all these methods can be measured using phase transition plots, which use grayscale colors to depict how likely a certain kind of low-rank tensors can be recovered by those algorithms for a range of different ranks from errors of varying sparsity. If the relative error $\textup{RSE} \leq 10^{-3}$, we declare the trial to be successful. Fig.~\ref{fig1} shows the phase transition plots of all algorithms on the third-order tensors of size $100 \times 100 \times 100$, where the $x$-axis corresponds to the Tucker rank $r_n, \forall n \in \mathds{N}$ changing from 2 to 50 with increment 2, and the $y$-axis denotes $\rho_n$ varying from $2\%$ to $50\%$ with increment $2\%$. For each setting, ten independent trials were run.

\begin{figure}[!t]
\centering
\subfigure[\large RPCA]{\includegraphics[width = 0.45\textwidth]{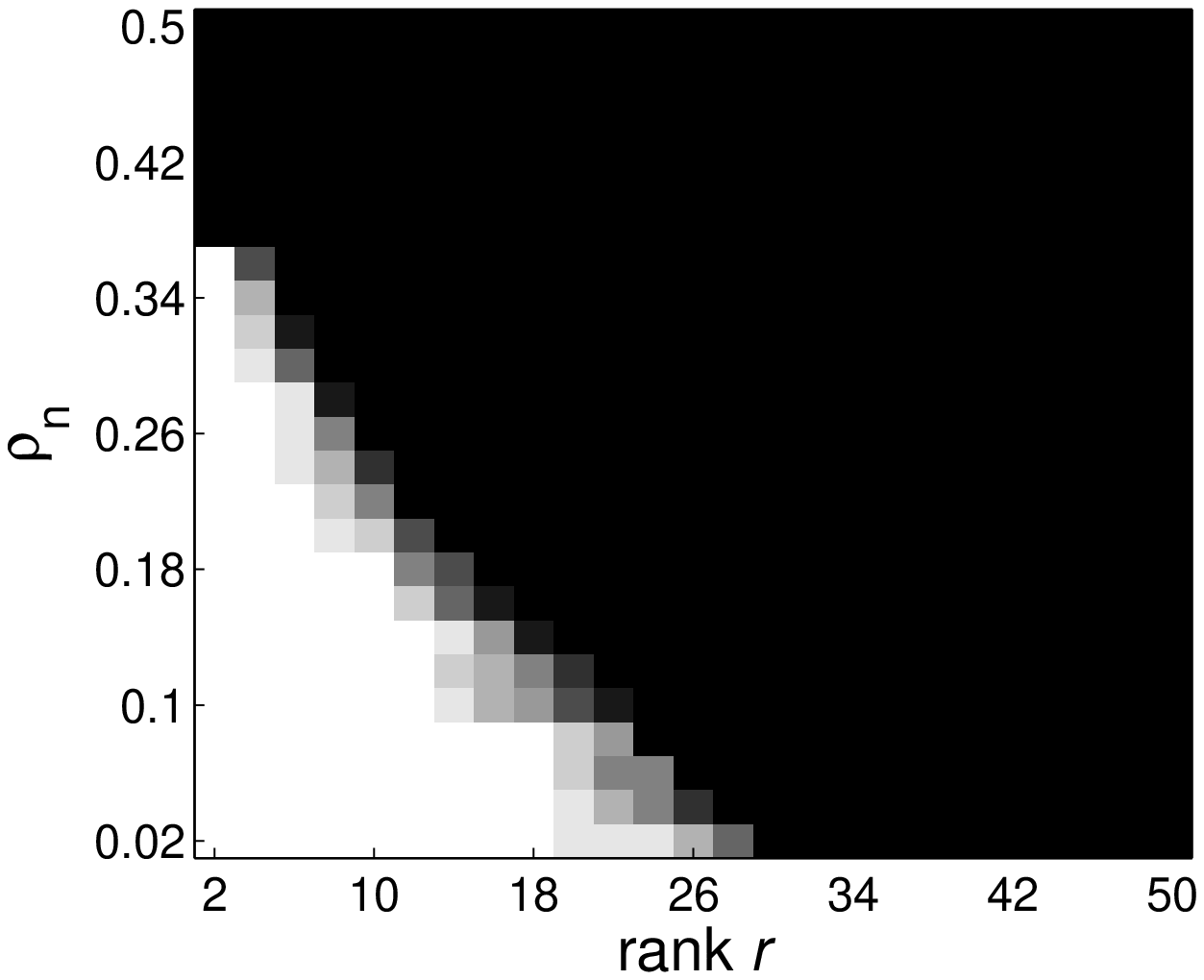}}
\subfigure[\large MALM]{\includegraphics[width = 0.45\textwidth]{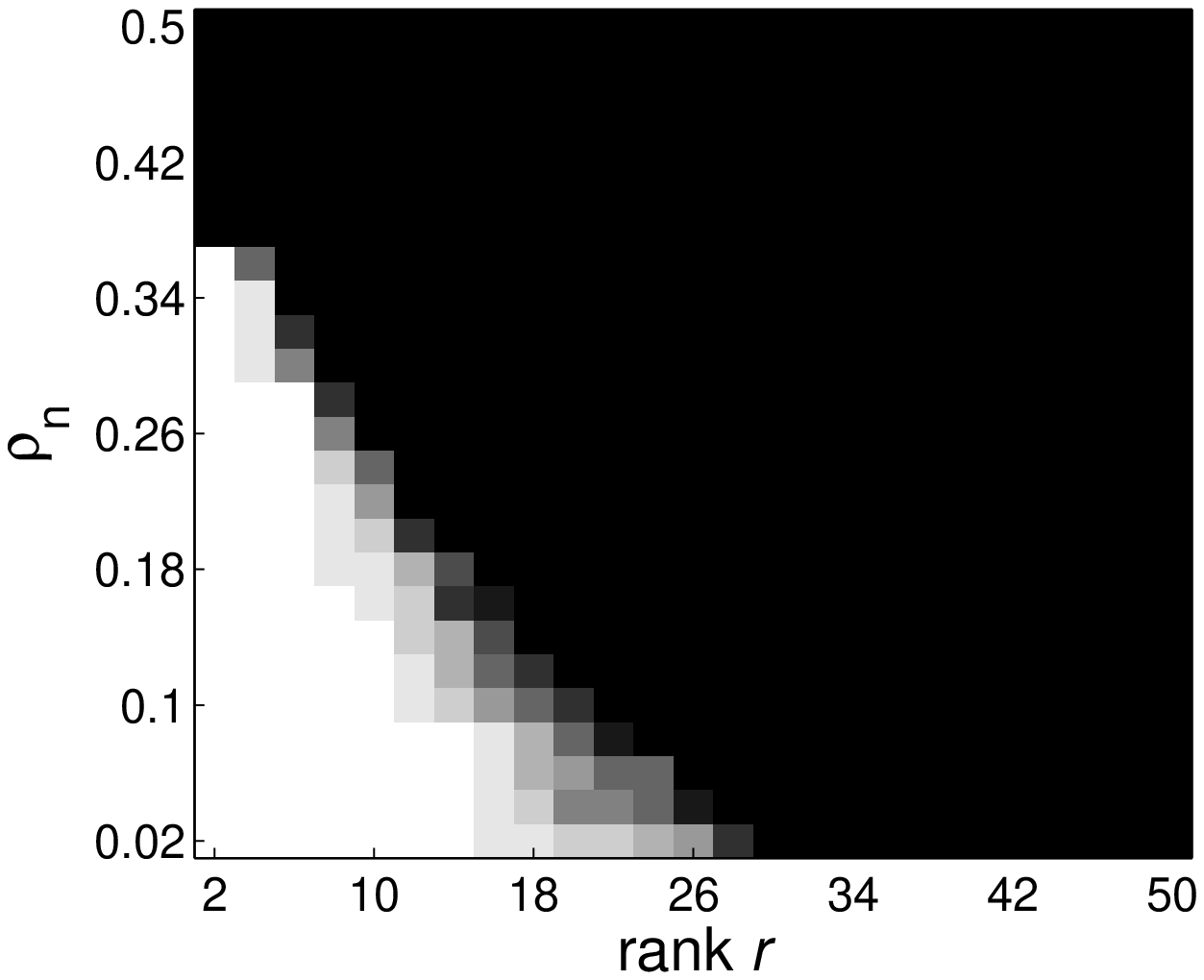}}\\
\subfigure[\large SNN]{\includegraphics[width = 0.45\textwidth]{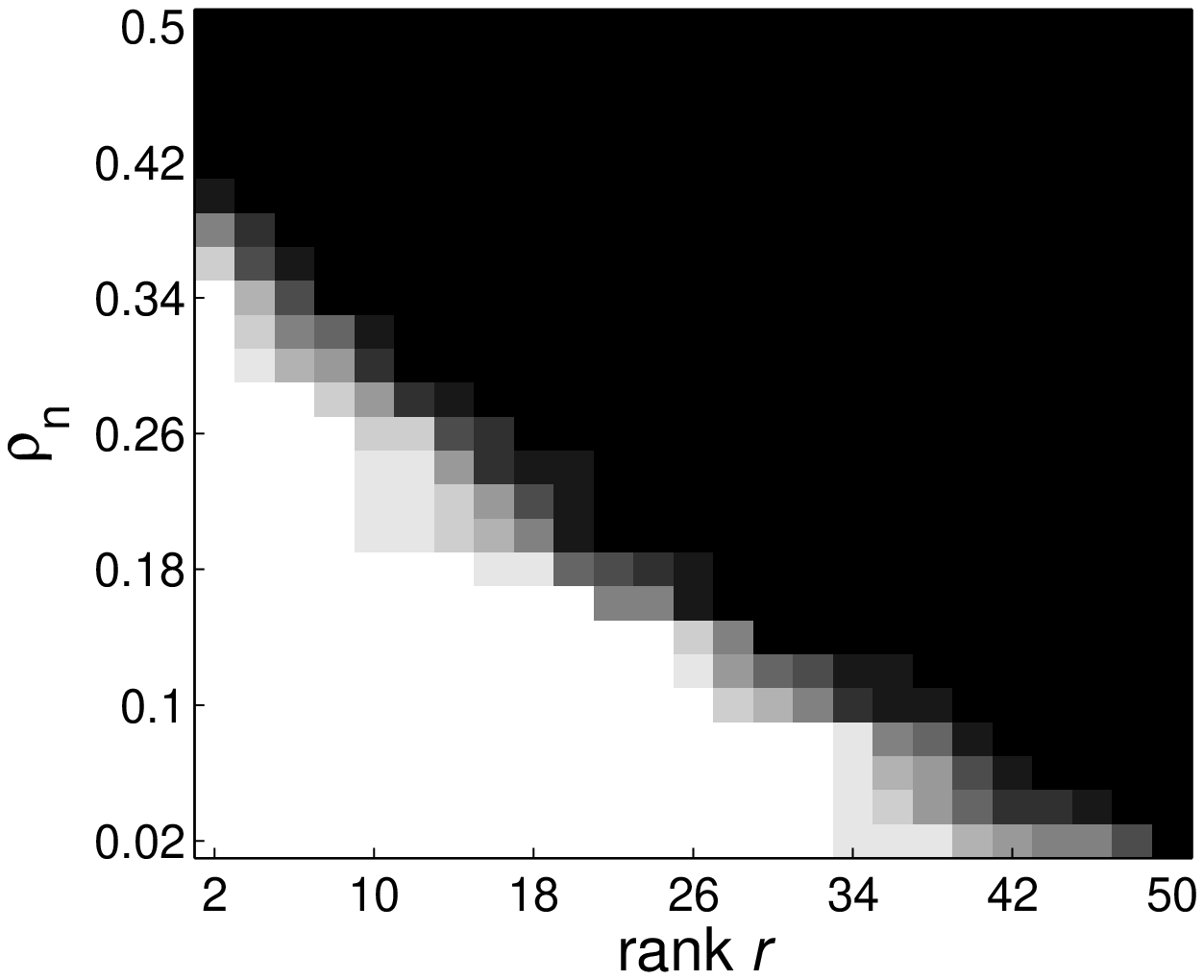}}
\subfigure[\large PASD]{\includegraphics[width = 0.45\textwidth]{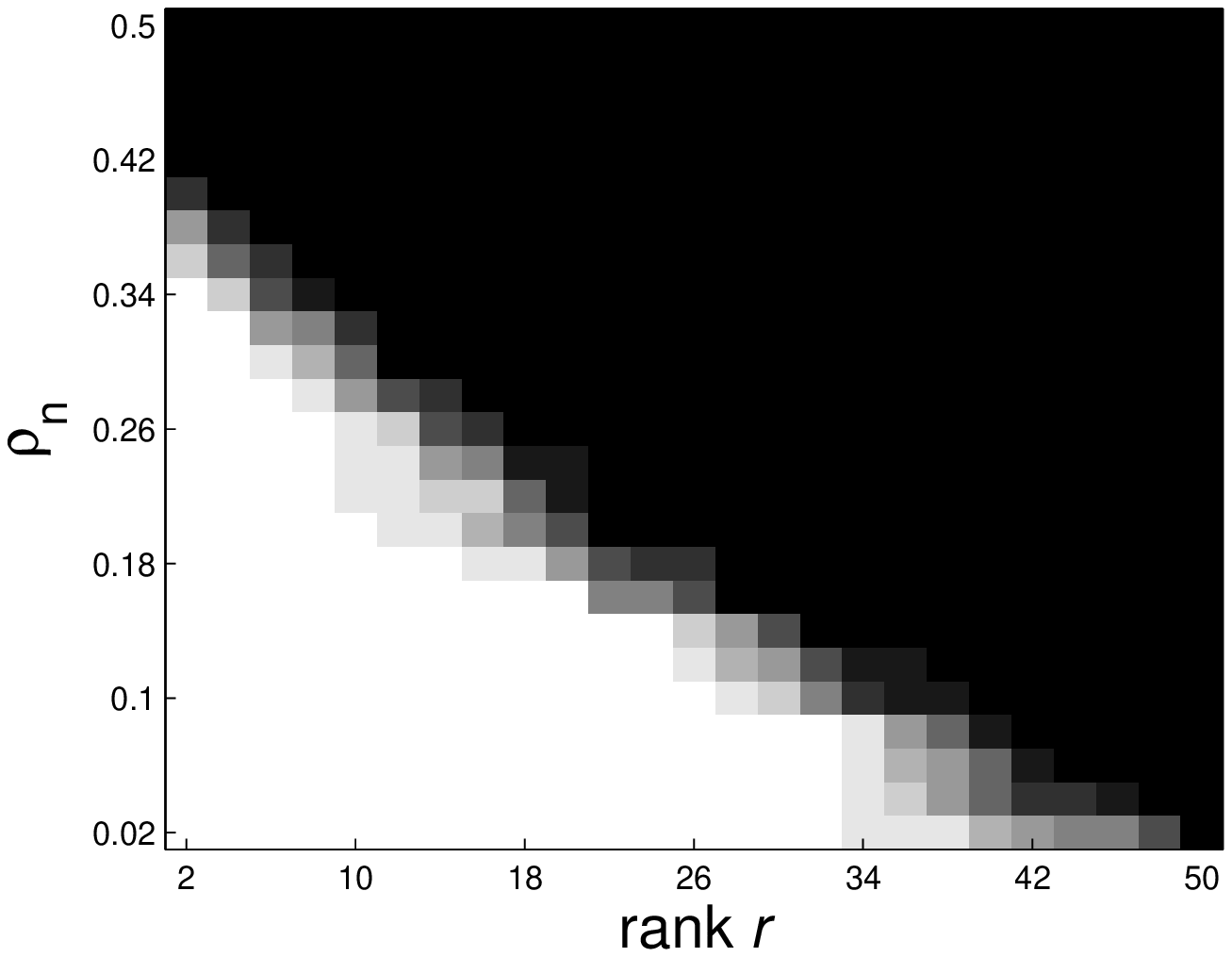}}
\caption{Phase transition plots on the third-order tensors. White region: 100\% success and black region: 0\% success in all experiments.}
\label{fig1}
\end{figure}

Next, we check the running time of all the methods on the 3-order tensors with varying sizes. As shown in Figure~\ref{fig2}, the running time of PASD increases much more slowly than those of the other approaches, which indicates that our PASD method is quite suitable for large-scale applications.

\begin{figure}[!t]
\centering
\subfigure[]{\includegraphics[width = 0.45\textwidth]{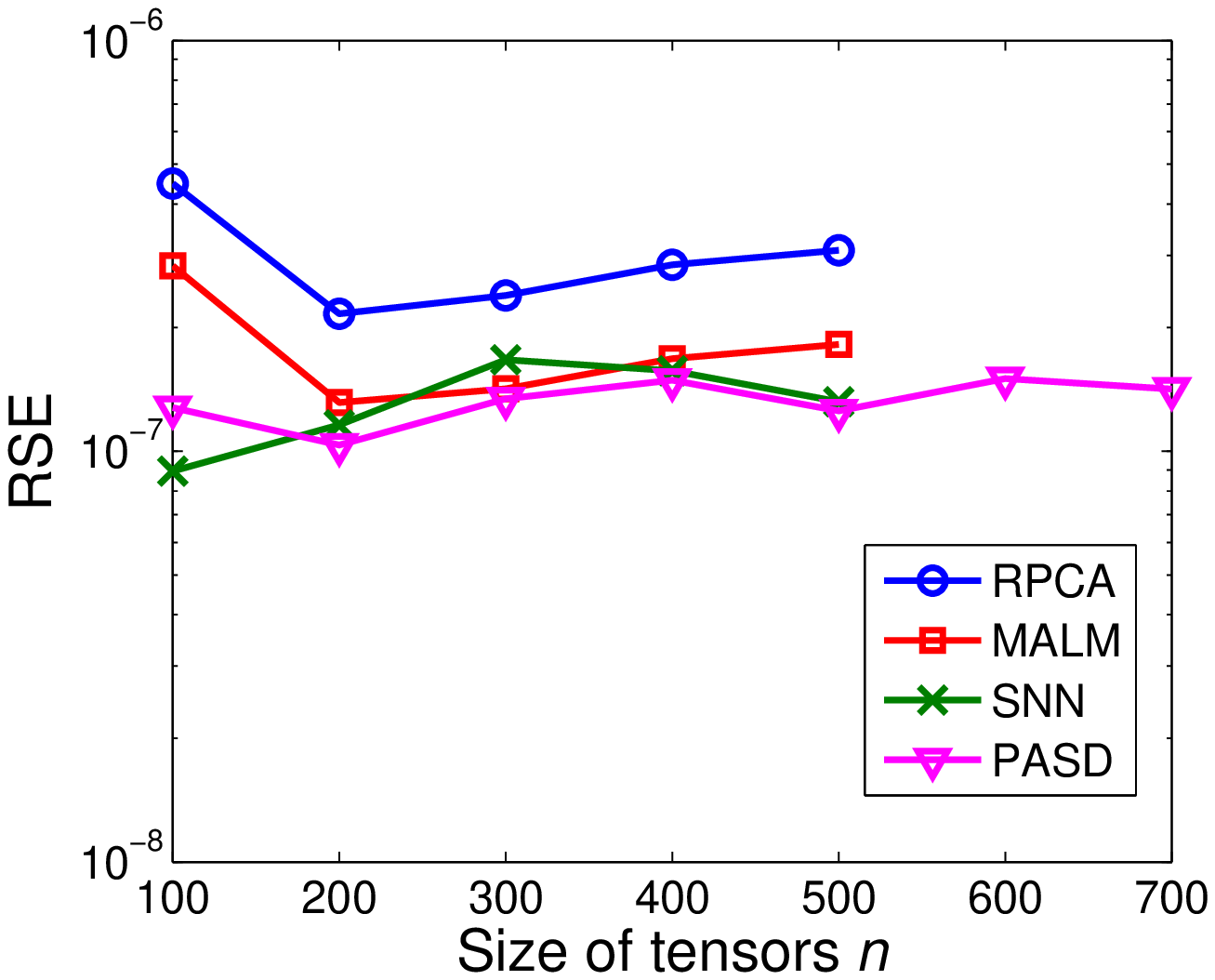}}
\subfigure[]{\includegraphics[width = 0.45\textwidth]{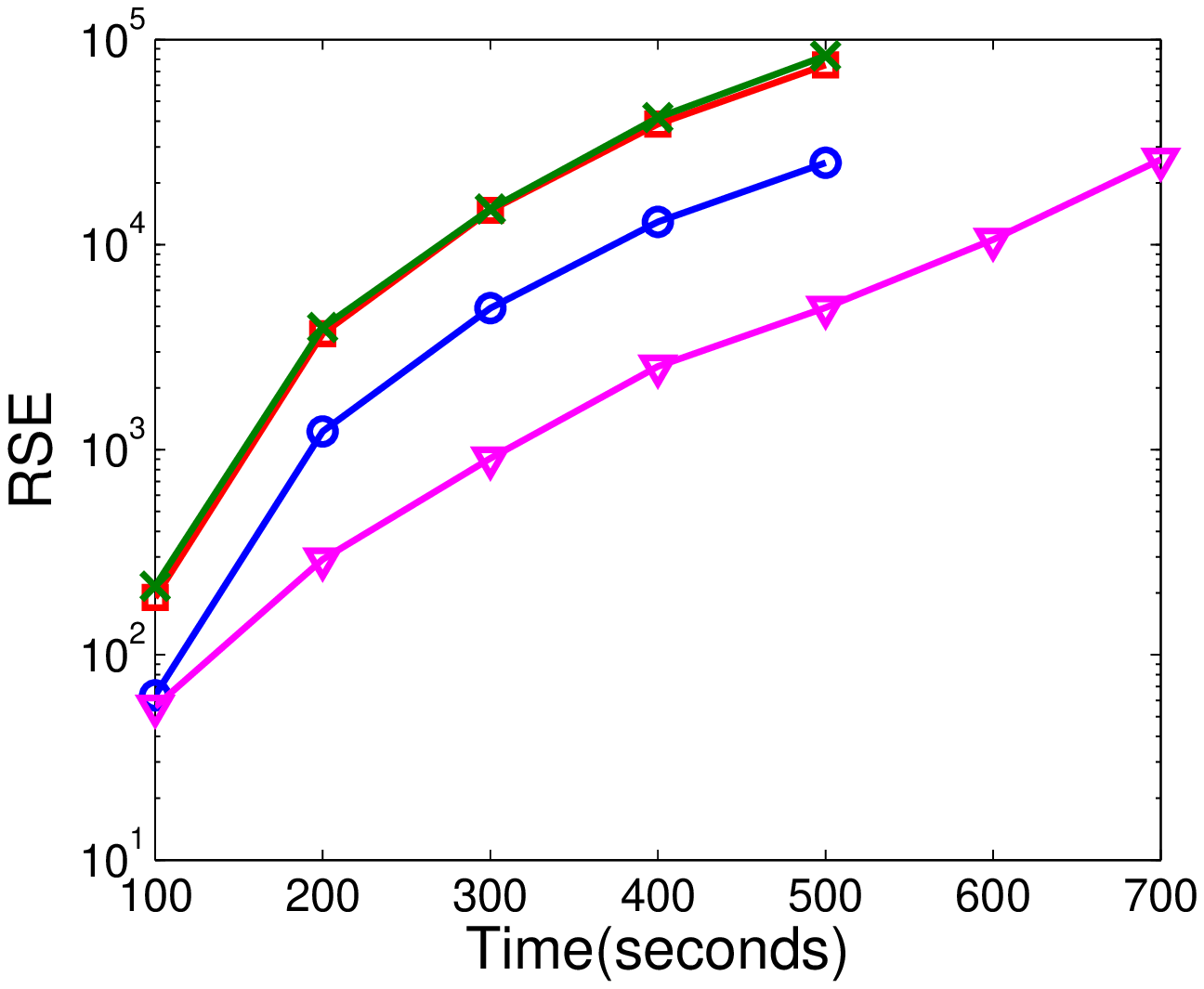}}
\caption{Comparison of all these methods in terms of RSE and computational
time (in logarithmic scale) on the third-order tensors by varying given tensor sizes.}
\label{fig2}
\end{figure}

\subsection{MRI Image Restoration}

In this experiment, we compare our PASD method with other approaches on the brain MRI image data, which is of size $181 \times 217 \times 181$ and is approximately low-rank~\cite{Liu2013}. We randomly choose $\rho_n$ percentage of pixels for each image to be corrupted by random values in [0, 255], where $\rho_n$ varies from 5\% to 30\%. We employ the Peak Signal to Noise Ratio (PSNR) to measure the difference between original image and the images recovered by various methods. For a specific $\rho_n$, the experiment is repeated 10 times and the average results are reported in Table~\ref{tab2}, where the parameters for all the methods are set as in the synthetic experiments and the upper bound of Tucker ranks are chosen as $R_n = 40, \forall n \in \mathds{N}$ for PASD.

\begin{table}
\caption{Average PSNR and running time (seconds) comparison on brain MRI data.}
\label{tab2}
\vspace{0.5em}
\centering
\begin{tabular}{c|cccccccc}
\hline
$\rho_n$ & \multicolumn{2}{c}{RPCA} & \multicolumn{2}{c}{MALM} & \multicolumn{2}{c}{SNN}  & \multicolumn{2}{c}{PASD}\\
& PSNR & Time &  PSNR & Time &  PSNR & Time &  PSNR & Time\\
\hline
0.05 & 47.96 & 529.91 & 48.07 & 1282.10 & 56.23 & 1482.00 & 56.22 & 324.97\\
0.10 & 46.89 & 517.95 & 47.02 & 1386.86 & 56.05 & 1281.08 & 56.02 & 311.23\\
0.15 & 45.07 & 392.70 & 45.27 & 1077.11 & 55.64 & 1069.74 & 55.58 & 270.69\\
0.20 & 42.62 & 399.76 & 42.97 & 977.63 & 54.76 & 1052.47 & 54.70 & 267.53\\
0.25 & 39.73 & 285.75 & 40.02 & 828.34 & 52.39 & 1047.31 & 52.36 & 271.87\\
0.30 & 36.59 & 289.81 & 36.81 & 837.23 & 48.69 & 899.38 & 48.74 & 264.41\\
\hline
\end{tabular}
\end{table}

\section{Conclusions}

In this paper, we propose a scalable and efficient method for the TRPCA problem. Considering that the heavy computational cost in the existing approaches are all stemming from the multiple SVDs conducted in each iteration, we split the unfoldings along each mode of the tensor into a columnwise orthonormal matrix (active subspace) and another small-size matrix. Such a transformation seems somewhat absurd, since it reformulate a convex optimization problem as a nonconvex one that is much more difficult to solve in general. But this reformulation indeed  allows us to replace the trace norm minimizations with large size by those involved some smaller-size matrices, and thus to reduce the computational complexity from $O(NI^{N+1})$ to $O(NRI^N)$ in each iteration. Therefore, our algorithm can scale pretty well to large-scale applications. The experiments show that our algorithm outperforms the state-of-the-art approaches in terms of both accuracy and efficiency. We expect that our PASD method can shed light on the development of new scalable algorithms for the problem of low-rank tensor recovery.

\bibliographystyle{plain}
\bibliography{2017_scalableTRPCA}

\appendix

\section{Proof of Lemma \ref{lem1}}
\label{appA}

To prove the boundedness of the sequences, we first introduce the following lemma.
\begin{lemma}\textup{\cite{Lin2009}}
Let $\mathcal{H}$ be a real Hilbert space endowed with an inner product $\langle \cdot, \cdot \rangle$ and a corresponding norm $\|\cdot\|$, and $\bm{y} \in \partial\|\bm{x}\|$, where $\partial\|\cdot\|$ denotes the subgradient. Then $\|\bm{y}\|^{\ast} = 1$ if $\bm{x} \neq \bm{0}$, and $\|\bm{y}\|^{\ast} \leq 1$ if $\bm{x}=\bm{0}$, where $\|\cdot\|^{\ast}$ is the dual norm of the norm $\|\cdot\|$.
\label{lemmaS1}
\end{lemma}

As mentioned in the paper, we propose an ADMM scheme to circularly minimize $\mathcal{L}_{\mu}$ with respect to $\{U_n\}$, $\{V_n\}$, $\ten{E}$ and update $\ten{Y}_n$ as follow.
\begin{align}
\label{eq5}\min_{\{U_n\}} & \,\, \mathcal{L}_{\mu^k}(U_1 \cdots U_N, V^k_1 \cdots V^k_N, \ten{E}^k, \ten{Y}^k_1 \cdots \ten{Y}^k_N) \\
& \textup{s.t.}, \,\,\,U_n \in \St(I_n, R_n), \nonumber\\
\label{eq6}\min_{\{V_n\}} & \,\, \mathcal{L}_{\mu^k}(U^{k+1}_1 \cdots U^{k+1}_N, V_1 \cdots V_N, \ten{E}^k, \ten{Y}^k_1 \cdots \ten{Y}^k_N), \\
\label{eq7}\min_{\ten{E}} & \,\, \mathcal{L}_{\mu^k}(U^{k+1}_1 \cdots U^{k+1}_N, V^{k+1}_1 \cdots V^{k+1}_N, \ten{E}, \ten{Y}^k_1 \cdots \ten{Y}^k_N),\\
\label{eq8} \ten{Y}^{k+1}_n &  = \ten{Y}^k_n + \mu^k(\ten{T} - \fold_n(U^{k+1}_n V^{k+1}_n) - \ten{E}^{k+1}).
\end{align}

\begin{proof}
For convenience, we denote $\ser{U}^k = (U^k_1, \cdots, U^k_N)$, $\ser{V}^k = (V^k_1, \cdots, V^k_N)$ and $\ser{Y}^k = (\ten{Y}^k_1, \cdots, \ten{Y}^k_N)$. The first order optimal condition of problem (\ref{eq7}) with respect to $\ten{E}^{k+1}$ is
\begin{equation}
\bm{0} \in \partial_{\ten{E}^{k+1}} \mathcal{L}_{\mu^k}(\ser{U}^{k+1}, \ser{V}^{k+1}, \ten{E}^{k+1}, \ser{Y}^k), \nonumber
\end{equation}
i.e.,
\begin{equation}
\sum_{n=1}^N \left(\ten{Y}^k_n + \mu^k(\ten{T} - \fold_n(U^{k+1}_n V^{k+1}_n) - \ten{E}^{k+1})\right) \in \partial \|\ten{E}^{k+1}\|_1, \nonumber
\end{equation}
which implies that
\begin{equation}
\sum_{n=1}^N \ten{Y}^{k+1}_n \in \partial \|\ten{E}^{k+1}\|_1.\nonumber
\end{equation}
By Lemma~\ref{lemmaS1}, we have
\begin{equation}
\Big\|\sum_{n=1}^N \ten{Y}^{k+1}_n\Big\|_{\infty} \leq 1.
\label{eqS1}
\end{equation}
Hence, the sequence $\{\ten{Y}^k_n\}, \forall n \in \mathds{N}$ are all bounded.

The optimality of $U^{k+1}_n$ directly leads us to that
\begin{align}
\|\bar{\ten{Y}}^{k+1}_n\|_F & = \|\ten{Y}^k_n + \mu^k(\ten{T} - \fold_n(U^{k+1}_n V^{k}_n) - \ten{E}^k)\|_F \nonumber\\
& \leq \|\ten{Y}^k_n + \mu^k(\ten{T} - \fold_n(U^{k}_n V^{k}_n) - \ten{E}^k)\|_F \nonumber\\
& = \|(1 + \rho) \ten{Y}^k_n - \rho \ten{Y}^{k-1}_n\|_F
\label{eqS2}
\end{align}
So $\{\bar{\ten{Y}^{k}_n}\}, \forall n \in \mathds{N}$ are bounded due to the boundedness of $\{\ten{Y}^k_n\}, \forall n \in \mathds{N}$.

The first-order optimal condition for problem (\ref{eq6}) with respect to $V^{k+1}_n$ is given by
\begin{equation}
\bm{0} \in \partial_{V^{k+1}_n} \mathcal{L}_{\mu^k} (\ser{U}^{k+1}, \ser{V}^{k+1}, \ten{E}^k, \ser{Y}^k), \nonumber
\end{equation}
which gives that
\begin{equation}
(U^{k+1}_n)^T \hat{\ten{Y}}^{k+1}_{n, (n)} \in \lambda_n \partial \|V^{k+1}_n\|_{\ast} \nonumber.
\end{equation}
By Lemma~\ref{lemmaS1}, we know that
\begin{equation}
\|(U^{k+1}_n)^T \hat{\ten{Y}}^{k+1}_{n, (n)}\| \leq \lambda_n,
\label{eqS3}
\end{equation}
and thus $\left\{(U^{k+1}_n)^T \hat{\ten{Y}}^{k+1}_{n, (n)}\right\}, \forall n \in \mathds{N}$ are bounded. Let $(U^{k+1}_n)^{\perp}$ denote the orthogonal complement of $U^{k+1}_n$, and we can easily check that
\begin{equation}
\left((U^{k+1}_n)^{\perp}\right)^T \hat{\ten{Y}}^{k+1}_{n, (n)}  = \left((U^{k+1}_n)^{\perp}\right)^T \bar{\ten{Y}}^{k+1}_{n, (n)} \nonumber
\end{equation}
which immediately implies
\begin{equation}
\left\|\left((U^{k+1}_n)^{\perp}\right)^T \hat{\ten{Y}}^{k+1}_{n, (n)}\right\|  = \left\|\left((U^{k+1}_n)^{\perp}\right)^T \bar{\ten{Y}}^{k+1}_{n, (n)}\right\| \leq \left\|\bar{\ten{Y}}^{k+1}_{n, (n)}\right\|.
\label{eqS4}
\end{equation}
Therefore, $\left\{\left((U^{k+1}_n)^{\perp}\right)^T \hat{\ten{Y}}^{k+1}_{n, (n)}\right\}, \forall n \in \mathds{N}$ are bounded. According to these two facts, $\{\hat{\ten{Y}}^{k}_n\}, \forall n \in \mathds{N}$ are bounded as well.

By the iteration procedure of Algorithm 1, we have
\begin{align}
\mathcal{L}_{\mu^k}(\ser{U}^{k+1}, \ser{V}^{k+1}, \ten{E}^{k+1}, \ser{Y}^k) & \leq \mathcal{L}_{\mu^k}(\ser{U}^{k+1}, \ser{V}^{k+1}, \ten{E}^k, \ser{Y}^k) \nonumber\\
& \leq \mathcal{L}_{\mu^k}(\ser{U}^{k+1}, \ser{V}^k, \ten{E}^k, \ser{Y}^k) \nonumber\\
& \leq \mathcal{L}_{\mu^k}(\ser{U}^k, \ser{V}^k, \ten{E}^k, \ser{Y}^k) \nonumber\\
& = \mathcal{L}_{\mu^{k-1}}(\ser{U}^k, \ser{V}^k, \ten{E}^k, \ser{Y}^{k-1}) + \sum_{n=1}^N \frac{\mu^{k-1} + \mu^k}{2(\mu^{k-1})^2} \|\ten{Y}^k_n - \ten{Y}^{k-1}_n\|^2_F.
\label{eqS5}
\end{align}
Note that $\mu^k = \rho \mu^{k-1}$ and we have
\begin{equation}
\sum_{k=1}^{\infty} \frac{\mu^{k-1} + \mu^k}{2(\mu^{k-1})^2}  = \frac{\rho(\rho + 1)}{2\mu^0(\rho - 1)} < \infty. \nonumber
\end{equation}
Hence, $\left\{\mathcal{L}_{\mu^k}(\ser{U}^k, \ser{V}^k, \ten{E}^k, \ser{Y}^{k-1})\right\}$ is upper bounded due to boundness of $\{\ten{Y}^k_n\}, \forall n \in \mathds{N}$. Then,
\begin{equation}
\sum_{n=1}^N \lambda_n \|V^k_n\|_{\ast} + \lambda \|\ten{E}^k\|_1 = \mathcal{L}_{\mu^k}(\ser{U}^k, \ser{V}^k, \ten{E}^k, \ser{Y}^k) - \frac{1}{2\mu^{k-1}} \sum_{n=1}^N \left(\|\ten{Y}^k_n\| - \|\ten{Y}^{k-1}_n\|\right)
\label{eqS6}
\end{equation}
is also upper bounded, which means that $\{V^k_n\}, \forall n \in \mathds{N}$ and $\{\ten{E}^k\}$ are bounded. Since $\|U^k_n V^k_n\|_{\ast} = \|V^k_n\|_{\ast}$, $\{U^k_n V^k_n\}, \forall n \in \mathds{N}$ are bounded as well.
\end{proof}

\section{Proof of Theorem \ref{the1}}
\label{appB}

\begin{proof}
(I) The boundedness of $\ten{Y}_n^k$, $\hat{\ten{Y}}^k_n$ and $\bar{\ten{Y}}^k_n$ and the fact $\lim\limits_{k \rightarrow \infty} \mu^k \rightarrow \infty$ imply that
\begin{equation}
\frac{\ten{Y}^{k+1}_n - \ten{Y}^k_n}{\mu^k} \rightarrow 0, \,\,\,\frac{\hat{\ten{Y}}^{k+1}_n - \ten{Y}^k_n}{\mu^k} \rightarrow 0, \,\,\, \frac{\bar{\ten{Y}}^{k+1}_n - \hat{\ten{Y}}^k_n}{\mu^k} \rightarrow 0, \,\,\, \forall n \in \mathds{N} \nonumber
\end{equation}
By the definitions of $\{\ten{Y}^k_n\}$, $\{\hat{\ten{Y}}^k_n\}$ and $\{\bar{\ten{Y}}^k_n\}$, we have that
\begin{align}
\ten{E}^{k+1} - \ten{E}^k & = \frac{\hat{\ten{Y}}^{k+1}_n - \ten{Y}^{k+1}_n}{\mu^k}, \,\,\, \forall n \in \mathds{N}, \nonumber\\
V^{k+1}_n - V^k_n & = \frac{(U^{k+1}_n)^T\left(\bar{\ten{Y}}^{k+1}_{n, (n)} - \hat{\ten{Y}}^{k+1}_{n, (n)}\right)}{\mu^k}, \,\,\, \forall n \in \mathds{N}, \nonumber\\
U^{k+1}_n V^{k+1}_n - U^k_n V^k_n & = \frac{(1+\rho)\ten{Y}^k_{n, (n)} - \left(\hat{\ten{Y}}^{k+1}_n + \rho\ten{Y}^{k+1}_{n, (n)}\right)}{\mu^k}, \,\,\, \forall n \in \mathds{N}. \nonumber
\end{align}
Therefore, the sequences $\{V^k_n\}, \{U^k_n V^k_n\}, \forall n \in \mathds{N}$ and $\{\ten{E}^k\}$ are Cauchy sequences.
$\\\\$
(II) It is easy to check that
\begin{equation}
\ten{T}  - \fold_n(U^{k+1}_n V^{k+1}_n) - \ten{E}^{k+1} = \frac{\ten{Y}^{k+1}_n - \ten{Y}^k_n}{\mu^k}.
\end{equation}
By the boundness of $\{\ten{Y}^k_n\}, \forall n \in \mathds{N}$ and $\lim\limits_{k\rightarrow\infty} \mu^k \rightarrow \infty$, we have that
\begin{equation}
\lim_{k \rightarrow \infty} \ten{T} - \fold_n(U^{k+1}_n V^{k+1}_n) - \ten{E}^{k+1} \rightarrow 0,
\end{equation}
and thus $(U^k_n, V^k_n, \ten{E}^k), \forall n \in \mathds{N}$ approaches to a feasible solution.
\end{proof}

\section{Proof of Lemma 2}
\label{appC}

To prove Lemma 2, we need to introduce the following lemma.
\begin{lemma}~\textup{\cite{Liu2012}}
Let $X$, $Y$ and $Q$ be matrices of compatible dimensions. If $Q$ obeys $Q^T Q = I$ and $Y \in \partial \|X\|_{\ast}$, then $QY \in \partial \|X\|_{\ast}$.
\label{lemS2}
\end{lemma}

\begin{proof}\textbf{of Lemma 2} Let the skinny SVD of $G^k_n$ be $G^k_n = \hat{U}^k_n \hat{\Sigma}^k_n (\hat{V}^k_n)^T$, then it can be computed that
\begin{equation}
U^{k+1}_n = \hat{U}^k_n (\hat{V}^k_n)^T (V^k_n)^T. \nonumber
\end{equation}
Let the full SVD of $\hat{\Sigma}^k_n (\hat{V}^k_n)^T(V^k_n)^T$ be $\hat{\Sigma}^k_n (\hat{V}^k_n)^T(V^k_n)^T = \tilde{U}^k_n \tilde{\Sigma}^k_n (\tilde{V}^k_n)^T$. Note that $\tilde{U}^k_n$ and $\tilde{V}^k_n$ are orthogonal matrices, then we have that
\begin{equation}
U^{k+1}_n = \hat{U}^k_n \tilde{U}^k_n (\tilde{V}^k_n)^T, \nonumber
\end{equation}
which simply implies that
\begin{equation}
U^{k+1}_n (U^{k+1}_n)^T = \hat{U}^k_n \tilde{U}^k_n (\tilde{V}^k_n)^T \tilde{V}^k_n (\tilde{U}^k_n)^T (\hat{U}^k_n)^T = \hat{U}^k_n (\hat{U}^k_n)^T.\nonumber
\end{equation}
Hence,
\begin{align}
\hat{\ten{Y}}^{k+1}_{n, (n)} & = \mu^k\left(\Big(\ten{T}_{(n)} - \ten{E}^{k}_{(n)} +\frac{\ten{Y}^k_{n, (n)}}{\mu^k}\Big) - U^{k+1}_n (U^{k+1}_n)^T\Big(\ten{T}_{(n)} - \ten{E}^{k}_{(n)} +\frac{\ten{Y}^k_{n, (n)}}{\mu^k}\Big)\right) \nonumber\\
& = \mu^k (\hat{U}^{k}_n \hat{\Sigma}^{k}_n (\hat{V}^k_n)^T - U^{k+1}_n (U^{k+1}_n)^T\hat{U}^{k}_n \hat{\Sigma}^{k}_n (\hat{V}^k_n)^T) \nonumber\\
& = \mu^k (\hat{U}^{k}_n \hat{\Sigma}^{k}_n (\hat{V}^k_n)^T - \hat{U}^{k}_n (\hat{U}^{k}_n)^T\hat{U}^{k}_n \hat{\Sigma}^{k}_n (\hat{V}^k_n)^T) \nonumber\\
& = \mu^k (\hat{U}^{k}_n \hat{\Sigma}^{k}_n (\hat{V}^k_n)^T - \hat{U}^{k}_n \hat{\Sigma}^{k}_n (\hat{V}^k_n)^T) = 0, \nonumber
\end{align}
i.e.,
\begin{equation}
\hat{\ten{Y}}^{k+1}_{n, (n)} = U^{k+1}_n (U^{k+1}_n)^T \hat{\ten{Y}}^{k+1}_{n, (n)}.
\label{eqS11}
\end{equation}

According to (\ref{eqS3}) and Lemma~\ref{lemS2}, we have
\begin{equation}
U^{k+1}_n (U^{k+1}_n)^T \hat{\ten{Y}}^{k+1}_{n, (n)} \in \lambda_n\partial \|U^{k+1}_n V^{k+1}_n\|_{\ast}, \,\,\, \textup{and thus}\,\,\, \hat{\ten{Y}}^{k+1}_{n, (n)} \in \lambda_n\partial\|U^{k+1}_n V^{k+1}_n\|_{\ast}.
\label{eqS12}
\end{equation}
Since (\ref{eqS1}) and (\ref{eqS12}) hold for any $k$, they naturally hold at $(\{U^{\ast}\}, \{V^{\ast}\}, \ten{E}^{\ast})$
\begin{equation}
\hat{\ten{Y}}^{\ast}_{n, (n)} \in \lambda_n\|U^{\ast}_n V^{\ast}_n\|_{\ast}, \quad \sum_{n=1}^N \ten{Y}^{\ast}_n \in  \partial\|\ten{E}^{\ast}\|_1. \label{eqS13}
\end{equation}

Given any feasible solution $(\{U_n\}, \{V_n\}, \ten{E})$ to problem~(3), by the convexity of nuclear norm and $\ell_1$ norm, we have that
\begin{align}
&\sum_{n=1}^N \lambda_n \|V_n\|_{\ast} + \|\ten{E}\|_1  = \sum_{n=1}^N \lambda_n \|U_n V_n\|_{\ast} + \|\ten{E}\|_1 \nonumber\\
& \geq \sum_{n=1}^N \left(\lambda_n \|U^{\ast}_n V^{\ast}_n\|_{\ast} + \langle \hat{\ten{Y}}^{\ast}_{n, (n)}, U_n V_n - U^{\ast}_n V^{\ast}_n \rangle\right) + \|\ten{E}^{\ast}\|_1 + \langle \sum_{n=1}^N \ten{Y}^{\ast}_n, \ten{E} - \ten{E}^{\ast} \rangle \nonumber\\
& = \sum_{n=1}^N \lambda_n \|U^{\ast}_n V^{\ast}_n\|_{\ast} + \|\ten{E}^{\ast}\|_1  + \sum_{n=1}^N \langle \ten{Y}^{\ast}_n - \hat{\ten{Y}}^{\ast}_n, \ten{E} - \ten{E}^{\ast}\rangle + \sum_{n=1}^N \langle \hat{\ten{Y}}^{\ast}_{n}, \fold_n(U_n V_n) + \ten{E} - \fold_n(U^{\ast}_n V^{\ast}_n) - \ten{E}^{\ast}\rangle \nonumber
\end{align}
By Theorem 1, we have that
\begin{equation}
\|\fold_n(U_n V_n) + \ten{E} - \fold_n(U^{\ast}_n V^{\ast}_n) - \ten{E}^{\ast}\|_{\infty} \leq \|\ten{T} - \fold_n(U^{\ast}_n V^{\ast}_n) - \ten{E}^{\ast}\|_{\infty} \leq \varepsilon, \nonumber
\end{equation}
which directly leads to
\begin{align}
&|\langle \hat{\ten{Y}}^{\ast}_n,  \fold_n(U_n V_n) + \ten{E} - \fold_n(U^{\ast}_n V^{\ast}_n) - \ten{E}^{\ast}\rangle| \nonumber\\
\leq & \|\hat{\ten{Y}}^{\ast}_n\|_{\infty} \|\fold_n(U_n V_n) + \ten{E} - \fold_n(U^{\ast}_n V^{\ast}_n) - \ten{E}^{\ast}\|_1 \nonumber\\
= & \|\hat{\ten{Y}}^{\ast}_{n, (n)}\|_{\infty} \|\fold_n(U_n V_n) + \ten{E} - \fold_n(U^{\ast}_n V^{\ast}_n) - \ten{E}^{\ast}\|_1 \nonumber\\
\leq & \|\hat{\ten{Y}}^{\ast}_{n, (n)}\| \|\fold_n(U_n V_n) + \ten{E} - \fold_n(U^{\ast}_n V^{\ast}_n) - \ten{E}^{\ast}\|_{\infty} \nonumber\\
\leq & \lambda_n I_n \prod_{m \neq n} I_m \varepsilon \nonumber
\end{align}
where $\|\hat{\ten{Y}}^{\ast}_{n, (n)}\| \leq \lambda_n$ is due to (\ref{eqS13}). Hence, we complete the proof.
\end{proof}

\section{Proof of Theorem 2}
\label{appD}

\begin{proof}
Note that $(\{U_n\} = \bm{0}, \{V_n\} = \bm{0}, \ten{E} = \ten{T})$ is feasible to (3) and let $(\{U^g_n\}, \{V^g_n\}, \ten{E}^g)$ be a globally optimal solution to (3), then we have
\begin{equation}
\|\ten{E}^g\|_1 \leq \sum_{n=1}^N \lambda_n \|V^g_n\|_{\ast} + \|\ten{E}^g\|_1 \leq  \|\ten{T}\|_1. \nonumber
\end{equation}
By the proof of Lemma 1, we have that $\ten{Y}_n^{\ast} = \ten{Y}_m^{\ast}, \forall n, m \in \mathds{N}, n \neq m$ almost surely since $\ten{Y}_n^{0} = \ten{Y}_m^{0} = \bm{0}$. Recall that $\|\sum\limits_{n=1}^N \ten{Y}_n^{k^{\ast}}\|_{\infty} \leq 1$, and we have that  $\|\ten{Y}_n^{k^{\ast}}\|_{\infty} \leq \frac{1}{N}, \forall n \in \mathds{N}$. So
\begin{equation}
\|\ten{Y}_n ^{k^{\ast}}\|_F \nonumber = \|\ten{Y}_{n, (n)} ^{k^{\ast}}\|_F \leq \sqrt{I_n \prod_{m \neq n}I_m} \|\ten{Y}_{n, (n)} ^{k^{\ast}}\|_{\infty} = \sqrt{I_n \prod_{m \neq n}I_m} \|\ten{Y}_{n} ^{k^{\ast}}\|_{\infty} \leq \sqrt{I_n \prod_{m \neq n}I_m} \frac{1}{N}
\end{equation}
holds and thus $\ten{E}^{\ast}$ is bounded by
\begin{align}
\|\ten{E}^{\ast}\|_1 & \leq \sum_{n=1}^N \lambda_n \|V^{\ast}_n\|_{\ast} + \|\ten{E}^{\ast}\|_1 \nonumber\\
& \leq \mathcal{L}_{\mu^{k^\ast}} (\ser{U}^{k^\ast+1}, \ser{V}^{k^\ast+1}, \ten{E}^{k^\ast+1}, \ser{Y}^{k^\ast+1}) + \sum_{n=1}^N \frac{\|\ten{Y}^{k^\ast}_n\|^2_F}{2\mu^{k^\ast}} \nonumber\\
& \leq \sum_{n=1}^N \frac{1}{\mu^0 N^2} I_n \prod_{m \neq n} I_m \left(\frac{\rho(1+\rho)}{\rho -1} + \frac{1}{2\rho^{k^\ast}}\right).
\end{align}
Hence, $\|\ten{E}^g - \ten{E}^{\ast}\|_1 \leq \|\ten{E}^g\|_1 + \|\ten{E}^{\ast}\|_1 \leq c$. By Lemma 2, we have
\begin{align}
f^g  = \sum_{n=1}^N \lambda_n \|V^g_n\|_{\ast} + \|\ten{E}^g\|_1 & \geq \sum_{n=1}^N \lambda_n \|V^{\ast}_n\|_{\ast} +  \|\ten{E}\|_1 + \sum_{n=1}^N \langle \ten{Y}^{\ast}_n - \hat{\ten{Y}}^{\ast}_n, \ten{E}^g - \ten{E}^{\ast} \rangle \nonumber\\
& \geq f^{\ast} - \|\sum_{n=1}^N \ten{Y}^{\ast}_n - \sum_{n=1}^N \hat{\ten{Y}}^{\ast}_n\|_{\infty} \|\ten{E}^g - \ten{E}^{\ast}\|_1 - \sum_{n=1}^N \lambda_n I_n \prod_{m \neq n} I_m \varepsilon \nonumber\\
& \geq f^{\ast} - c\epsilon - \sum_{n=1}^N \lambda_n I_n \prod_{m \neq n} I_m \varepsilon  \nonumber
\end{align}
which complete the proof.
\end{proof}

\section{Proof of Theorem 3}
\label{appE}

\begin{proof}
By the convexity of problem (1) and the optimality of $(\ten{X}^0, \ten{E}^0)$, it naturally follows that $f^0 \leq f^{\ast}$. Let $\ten{X}^0_{(n)} = U^0_n \Sigma^0_n (V^0_n)^T$ be the skinny SVD of mode-$n$ unfolding $\ten{X}^0_{(n)}$. Constructing $U'_n = U^0_n$, $V'_n = \Sigma^0_n (V^0_n)^T$ and $\ten{E}' = \ten{E}^0$, we have the following equality when $R_n \geq r_n$,
\begin{equation}
\ten{T} = \ten{X}^0 + \ten{E}^0 = \fold_n(U^0_n\Sigma^0_n(V^0_n)^T) + \ten{E}^0 = \fold_n(U'_n V'_n) + \ten{E}'
\end{equation}
i.e., $(\{U'_n\}, \{V'_n\}, \ten{E}')$ is a feasible solution to problem (3). By Theorem 2, we can conclude that
\begin{equation}
f^{\ast} - c\epsilon - \sum_{n=1}^N \lambda_n I_n \prod_{m \neq n} I_m \leq f^0. \nonumber
\end{equation}
For $R_n < r_n$, we decompose the skinny SVD of $\ten{X}^0_{(n)}$ as
\begin{equation}
\ten{X}^0_{(n)} = U^1_n \Sigma^1_n (V^1_n)^T + U^2_n \Sigma^2_n (V^2_n)^T, \nonumber
\end{equation}
where $U^1_n$ and $V^1_n$ (resp. $U^2_n$ and $V^2_n$) are the singular vectors associated with the $R_n$ largest singular values (resp. the rest singular values smaller than or equal to $\sigma^{R_n}_{(n)}$). With these notations, we have a feasible solution to problem (3) by constructing
\begin{equation}
U''_n  = U^1_n, \,\,\, V''_n = \Sigma^1_n (V^1_n)^T, \,\,\, \ten{E}'' = \ten{E}^0 + \fold_n(U^2_n \Sigma^2_n (V^2_n)^T). \nonumber
\end{equation}
By Theorem 2, we have that
\begin{align}
f^{\ast} - c\epsilon - \sum_{n=1}^N \lambda_n I_n \prod_{m \neq n} I_m \leq  f^g & \leq \sum_{n=1}^N \|V''_n\|^2_{\ast} +  \|\ten{E}''\|_1 \nonumber\\
& = \sum_{n=1}^N \lambda_n \Big(\|\Sigma^1_n (V^1_n)^T\|_{\ast} +  \|\ten{E}^0 + \fold_n(U^2_n \Sigma^2_n (V^2_n)^T)\|_1\Big) \nonumber\\
& \leq \sum_{n=1}^N \lambda_n \Big(\|\ten{X}^0_{(n)}\|_{\ast} - \|\Sigma^2_n\|_{\ast} + \|\ten{E}^0 + \fold_n(U^2_n \Sigma^2_n (V^2_n)^T)\|_1\Big) \nonumber\\
& \leq f^0 + \sum_{n=1}^N \lambda_n \left(\|U^2_n \Sigma^2_n (V^2_n)^T\|_1 -\|\Sigma^2_n\|_{\ast}\right) \nonumber\\
& \leq f^0 + \sum_{n=1}^N \lambda_n \left(\sqrt{I_n \prod_{n \neq m} I_m}\|U^2_n \Sigma^2_n (V^2_n)^T\|_{\ast} - \|\Sigma^2_n\|_{\ast}\right) \nonumber\\
& = f^0 + \sum_{n=1}^N \lambda_n \left(\sqrt{I_n \prod_{n \neq m} I_m} - 1\right)\|\Sigma^2_n\|_{\ast} \nonumber\\
& \leq f^0 + \sum_{n=1}^N \lambda_n \left(\sqrt{I_n \prod_{n \neq m} I_m} - 1\right)\sigma^{R_n + 1}_{(n)} (r_n - R_n), \nonumber
\end{align}
which complete the proof.
\end{proof}

\end{document}